\documentclass[11pt]{article}
\usepackage{amsfonts}
\usepackage{latexsym}
\usepackage{amsmath}
\usepackage{amssymb}
\usepackage{amsthm}
\usepackage{amsmath}
\usepackage{verbatim}
\usepackage{hyperref}
\usepackage{enumerate}
\usepackage{bibspacing}

\newtheorem{thm}{Theorem}[section]
\newtheorem*{thm*}{Theorem}

\newtheorem{lemma}[thm]{Lemma}

\newtheorem*{prop*}{Proposition}
\newtheorem{proposition}[thm]{Proposition}

\newtheorem*{conj*}{Conjecture}

\newtheorem*{dfn*}{Definition}
\theoremstyle{definition}

\newtheorem*{rmk*}{Remark}
\newtheorem*{fact*}{Fact}

\theoremstyle{proof}

\DeclareMathOperator{\Conv}{\textnormal{Conv}}

\newcommand{\Hess}{\text{Hess}}

\newcommand{\brac}[1]{\left(#1\right)}
\newcommand{\scalar}[1]{\left \langle #1 \right \rangle}
\newcommand{\sscalar}[1]{\langle #1 \rangle}

\newcommand{\R}{\mathbb{R}}

\newcommand{\E}{\mathbb{E}}
\renewcommand{\P}{\mathbb{P}}

\newcommand{\F}{\mathcal{F}}
\newcommand{\G}{\mathcal{G}}

\newcommand{\I}{\mathcal{I}}

\newcommand{\A}{\mathbf{A}}
\newcommand{\B}{\mathbf{B}}
\newcommand{\BL}{\mathcal{BL}}
\newcommand{\f}{\mathbf{f}}
\newcommand{\g}{\mathbf{g}}
\newcommand{\eps}{\epsilon}

\newcommand{\Q}{\mathbf{Q}}

\newcommand{\Id}{\textrm{Id}}

\numberwithin{equation}{section}

\numberwithin{equation}{section}

\begin{document}

\renewcommand*{\thefootnote}{\fnsymbol{footnote}}

\author{Emanuel Milman\thanks{Technion Israel Institute of Technology, Department of Mathematics, Haifa 32000, Israel. Email: emilman@tx.technion.ac.il.}
}

\begingroup    \renewcommand{\thefootnote}{}    
    \footnotetext{The research leading to these results is part of a project that has received funding from the European Research Council (ERC) under the European Union's Horizon 2020 research and innovation programme (grant agreement No 101001677).}
\endgroup

\title{Gaussian Correlation via Inverse Brascamp--Lieb}

\date{\nonumber} 
\maketitle

\begin{abstract}
We give a simple alternative proof of Royen's Gaussian Correlation inequality by using (a slightly generalized version of) Nakamura--Tsuji's symmetric inverse Brascamp--Lieb inequality for even log-concave functions. We explain that this inverse inequality is in a certain sense a dual counterpart to the forward inequality of Bennett--Carbery--Christ--Tao and Valdimarsson, and that the log-concavity assumption therein cannot be omitted in general. 
\end{abstract}

\section{Introduction}

Our starting point in this note is the observation that centered Gaussians saturate the following correlation lower-bound. This phenomenon (for the corresponding upper bound) has its origins in the work of Brascamp and Lieb \cite{BrascampLieb-YoungInq}.
 Given a symmetric $n \times n$ matrix $A$, we denote $g_{A}(x) := \exp(-\frac{1}{2} \scalar{A x, x})$ for $x \in \R^n$; when $A > 0$ is positive-definite, this corresponds to a centered Gaussian (whose covariance, after normalizing $g_A$ to be a probability density, is $A^{-1}$). All functions below are assumed to be non-negative. Recall that a function $f : \R^n \rightarrow \R_+$ is called log-concave (respectively, log-convex) if $\log f : \R^n \rightarrow \R \cup \{-\infty\}$ is concave (respectively, convex). Let $n_1,\ldots,n_m$ be positive integers, set $N = \sum_{i=1}^m n_i$, and consider the orthogonal decomposition $\R^N = \oplus_{i=1}^m \R^{n_i}$. Given $x \in \R^N$, write $x = (x_1,\ldots,x_m)$ with $x_i \in \R^{n_i}$. 

\smallskip

\begin{thm}[Symmetric Inverse Brascamp--Lieb Inequality, after Nakamura--Tsuji \cite{NakamuraTsuji-InverseBrascampLieb}] \label{thm:intro-IBL}
Let $\bar Q$ be an $N \times N$ symmetric matrix (\textbf{of arbitrary signature}), and let $Q_i \geq 0$ denote positive semi-definite $n_i \times n_i$ matrices, for $i=1,\ldots,m$. Let $c_1,\ldots,c_m>0$. Then for all  (non-zero) \textbf{even log-concave} $h_i \in L^1(\R^{n_i},g_{Q_i}(x_i) dx_i)$, $i=1,\ldots,m$, we have:
\[
\frac{\int_{\R^N} e^{-\frac{1}{2}\scalar{\bar Q x,x}} \Pi_{i=1}^m h_i(x_i)^{c_i} dx}{\Pi_{i=1}^m \brac{\int_{\R^{n_i}} e^{-\frac{1}{2} \scalar{Q_i x_i, x_i}} h_i(x_i) dx_i}^{c_i} } \geq \inf_{\substack{ A_i \geq 0\\A_i+Q_i > 0} } \frac{\int_{\R^N} e^{-\frac{1}{2} \scalar{\bar Q x,x}} \Pi_{i=1}^m g_{A_i}(x_i)^{c_i} dx}{\Pi_{i=1}^m \brac{\int_{\R^{n_i}} e^{-\frac{1}{2} \scalar{Q_i x_i, x_i}} g_{A_i}(x_i) dx_i}^{c_i} } ,
\]
where the infimum is over $n_i \times n_i$ symmetric matrices $A_i$, $i=1,\ldots,m$. 
\end{thm}

Theorem \ref{thm:intro-IBL} was recently established for $Q_i = 0$ by Nakamura and Tsuji in \cite[Theorem 1.3]{NakamuraTsuji-InverseBrascampLieb}. The extended version for general $Q_i \geq 0$ above is obtained by a straightforward repetition of their proof, which we verify in Section \ref{sec:IBL}. 
We refer to \cite{NakamuraTsuji-InverseBrascampLieb} for a thorough discussion and comparison with the original inequalities of Brascamp--Lieb and Lieb \cite{BrascampLieb-YoungInq, Lieb-MultiDimBL} (see also Carlen--Lieb--Loss \cite{CarlenLiebLoss-EntropyOnSn}, Bennett--Carbery--Christ--Tao \cite{BCCT-BrascampLieb}), the reverse Brascamp--Lieb inequality of Barthe \cite{Barthe-ReverseBL-CRAS,Barthe-ReverseBL} (see also Barthe--Cordero-Erausquin \cite{BartheCordero-InverseBLviaSemiGroup}, Valdimarsson \cite{Valdimarsson-GenCaffarelli} and Barthe--Huet \cite{BartheHuet}), the inverse Brascamp--Lieb inequalities of Chen--Dafnis--Paouris \cite{CDP-ReverseHolder} and Barthe--Wolff \cite{BartheWolff-InverseBrascampLieb}, and the unified forward-reverse Brascamp--Lieb inequalities of Liu--Courtade--Cuff--Verd\'u \cite{LCCV-BrascampLieb,CourtadeLiu-BrascampLieb}. See also \cite{CarlenCordero-SubadditivityOfEntropy,BCLM-BrascampLiebInqs, Boroczky-BrascampLiebSurvey} and the references therein for further related results. 
When $Q_i=0$ and $\bar Q=0$, scale invariance clearly implies that one should assume $\sum_{i=1}^m c_i n_i = N$ in order for the left-hand-side to have any chance of being bounded away from $0$ and $\infty$, but this is not the case for general $Q_i$ or $\bar Q$. The above mentioned works treat more general surjective linear transformations $P_i : \R^N \rightarrow \R^{n_i}$ with $\cap_{i=1}^m \ker P_i = \{ 0\}$, but we will only require and treat the case when $P_i x := x_i$ for $x = (x_1,\ldots,x_m)$, as in \cite{NakamuraTsuji-InverseBrascampLieb}. Here and below $Q_1 \oplus \ldots \oplus Q_m$ denotes the block matrix $P^*_1 Q_1 P_1 +  \ldots + P^*_m Q_m P_m$. 

We remark that some additional partial cases of Theorem \ref{thm:intro-IBL} were previously known in the literature. Denoting 
\begin{equation} \label{eq:Q}
Q := \bar Q - c_1 Q_1 \oplus \ldots \oplus c_m Q_m,
\end{equation}
 the case that $Q=0$ and general $Q_i > 0$ without requiring any evenness is contained in \cite[Theorem 1.4 (2)]{Valdimarsson-GenCaffarelli} of Valdimarsson, and the case when $Q \leq 0$  without requiring evenness nor log-concavity of $h_i$ on the left but also without requiring that $A_i \geq 0$ on the right is contained in \cite[Theorem 1.4]{BartheWolff-InverseBrascampLieb} of Barthe--Wolff (see also Chen--Dafnis--Paouris \cite{CDP-ReverseHolder} for the particular ``geometric case" when in addition 
  $P_i \bar Q^{-1} P_i^* = Q_i^{-1}$ for all $i$ and $\bar Q > 0$); 
however, neither of these assumptions fits our purposes.

\subsection{Log-concavity assumption} 

A noteworthy feature of the Nakamura--Tsuji formulation is that, contrary to the inverse Brascamp--Lieb inequalities in \cite{CDP-ReverseHolder,BartheWolff-InverseBrascampLieb}, there is no restriction on the signature of $Q$. 
As explained in \cite{BartheWolff-InverseBrascampLieb,NakamuraTsuji-InverseBrascampLieb}, the price for this flexibility is that the test functions $f_i$ cannot be arbitrarily translated, and are therefore assumed to be even. In \cite{NakamuraTsuji-InverseBrascampLieb}, Nakamura and Tsuji further discuss whether the additional assumption of log-concavity is also natural (as it is absent from the formulations in \cite{CDP-ReverseHolder,BartheWolff-InverseBrascampLieb,NakamuraTsuji-HypercontractivityBeyondNelson,NakamuraTsuji-VolumeProductUnderHeatFlow}), and write that it is actually reasonable to expect that it could be omitted. 
However, we shall see below that Theorem \ref{thm:intro-IBL} would be false without the log-concavity assumption (at least, if one allows considering $Q_i > 0$).
Furthermore, the log-concavity assumption is actually natural, 
as we shall now attempt to explain. 

\smallskip

Given a positive semi-definite $Q_i \geq 0$, we shall say $f_i$ is more log-concave (respectively, more log-convex) than $g_{Q_i}$ if $f_i = g_{Q_i} h_i$ for some log-concave (respectively, log-convex) function $h_i$. We shall say that $f_i$ is generated by $g_{Q_i}$ (``of $Q_i$-type" in the terminology of \cite{BCCT-BrascampLieb}) if $f_i$ is obtained as the convolution $g_{Q_i} \ast \mu_i$ for some finite (non-zero) Borel measure $\mu_i$ (note that such $f_i$ is always smooth). For smooth and positive $f_i$, clearly $f_i$ is more log-concave (respectively, more log-convex) than $g_{Q_i}$ iff $\Hess (-\log f_i) \geq (\leq) \, Q_i$, and it is easy to check that if $f_i$ is generated by $g_{Q_i}$ then it is more log-convex than $g_{Q_i}$ (e.g.~\cite[Lemma 16]{ChewiPooladian-GenCaffarelli}).
Clearly, recalling the definition (\ref{eq:Q}) of $Q$ and setting $B_i := A_i + Q_i$ and $f_i = g_{Q_i} h_i$, Theorem \ref{thm:intro-IBL} may be equivalently reformulated as follows (note that the equivalence crucially relies on the signatures of $\bar Q$ and $Q$ being arbitrary):

\begin{thm}[Theorem \ref{thm:intro-IBL} again] \label{thm:intro-IBL2}
Let $Q$ be an $N \times N$ symmetric matrix (\textbf{of arbitrary signature}), and let $Q_i \geq 0$ denote positive semi-definite $n_i \times n_i$ matrices, for $i=1,\ldots,m$. Let $c_1,\ldots,c_m>0$. Then for all (non-zero) \textbf{even} $f_i \in L^1(\R^{n_i})$ which are \textbf{more log-concave than $g_{Q_i}$}, we have:
\[
\frac{\int_{\R^N} e^{-\frac{1}{2}\scalar{Q x,x}} \Pi_{i=1}^m f_i(x_i)^{c_i} dx}{\Pi_{i=1}^m \brac{\int_{\R^{n_i}} f_i(x_i) dx_i}^{c_i} } \geq \inf_{\substack{B_i \geq Q_i\\ B_i > 0}} \frac{\int_{\R^N} e^{-\frac{1}{2} \scalar{Q x,x}} \Pi_{i=1}^m g_{B_i}(x_i)^{c_i} dx}{\Pi_{i=1}^m \brac{\int_{\R^{n_i}} g_{B_i}(x_i) dx_i}^{c_i} } . 
\]
\end{thm}

On the other hand, the following forward generalized Brascamp--Lieb inequality was established by Bennett--Carbery--Christ--Tao \cite[Corollaries 8.15 and 8.16]{BCCT-BrascampLieb} and Valdimarsson \cite[Theorem 1.4]{Valdimarsson-GenCaffarelli} whenever $f_i$ are generated by $g_{Q_i}$. Valdimarsson's proof is based on a generalized version of Caffarelli's contraction theorem \cite{CaffarelliContraction}, stating that if $f_i$ is generated by $g_{Q_i}$ and $h_i$ is more log-concave than $g_{Q_i^{-1}}$, then Brenier's optimal transport map $T = \nabla \varphi$ \cite{BrenierMap} pushing forward $\mu = a f_i(x) dx$ onto $\nu = b h_i(x) dx$ (both normalized to be probability measures) satisfies $0 \leq \Hess \; \varphi \leq Q_i$ \cite[Theorem 1.2]{Valdimarsson-GenCaffarelli}. In fact, the assumption that $f_i$ is generated by $g_{Q_i}$ can be relaxed to the assumption that $f_i$ is more log-convex than $g_{Q_i}$. This follows from an extension of Caffarelli's contraction theorem due to Kolesnikov \cite[Theorem 2.2]{KolesnikovContractionSurvey}, who showed that if $f^0_i$ is more log-convex than $g_{\Id}$ and $h^0_i$ is more log-concave than $g_{\Id}$ then the Brenier map $T_0 = \nabla \varphi_0$ pushing forward $\mu_0 = a f_i^0(x) dx$ onto $\nu_0 = b h_i^0(x) dx$ satisfies $0 \leq \Hess \; \varphi_0\leq \Id$; applying this to $\mu_0$ and $\nu_0$ given by the push-forwards of $\mu$ by $\sqrt{Q_i}$ and $\nu$ by $(\sqrt{Q_i})^{-1}$, respectively, and setting $\varphi = \varphi_0 \circ \sqrt{Q_i}$, it follows that $T = \nabla \varphi = \sqrt{Q_i} \circ T_0 \circ \sqrt{Q_i}$ is the Brenier map pushing forward $\mu$ onto $\nu$, and it satisfies $0 \leq  \Hess \; \varphi \leq Q_i$. See also the recent extensions to the case when $f_i$ is more log-convex than $g_{A_i}$ and $h_i$ is more log-concave than $g_{B_i^{-1}}$ for commuting and even general $A_i,B_i > 0$ due to Chewi--Pooladian \cite[Theorem 13]{ChewiPooladian-GenCaffarelli} and Gozlan--Sylvestre \cite[Theorem 5.4]{GozlanSylvestre-GeneralizedContractions}, respectively. Repeating Valdimarsson's argument from \cite{Valdimarsson-GenCaffarelli}, the following is thereby immediately deduced (see the proof of \cite[Lemma 8.13]{BCCT-BrascampLieb} to reduce to the case $Q = 0$):

\begin{thm}[After \cite{BCCT-BrascampLieb,Valdimarsson-GenCaffarelli,KolesnikovContractionSurvey}]
Let $Q \geq 0$ be an $N \times N$ \textbf{positive semi-definite} matrix, and let $Q_i \geq 0$ denote positive semi-definite $n_i \times n_i$ matrices, for $i=1,\ldots,m$. Let $c_1,\ldots,c_m>0$. Then for all (non-zero) $f_i \in L^1(\R^{n_i})$ which are \textbf{more log-convex than $g_{Q_i}$}  (but not necessarily even),
\[
\frac{\int_{\R^N} e^{-\frac{1}{2}\scalar{Q x,x}} \Pi_{i=1}^m f_i(x_i)^{c_i} dx}{\Pi_{i=1}^m \brac{\int_{\R^{n_i}} f_i(x_i) dx_i}^{c_i} } \leq \inf_{0 < B_i \leq Q_i} \frac{\int_{\R^N} e^{-\frac{1}{2}\scalar{Q x,x}} \Pi_{i=1}^m g_{B_i}(x_i)^{c_i} dx}{\Pi_{i=1}^m \brac{\int_{\R^{n_i}} g_{B_i}(x_i) dx_i}^{c_i} } . 
\]
\end{thm}
\noindent
Taking $Q_i = \Lambda \cdot \Id_{n_i}$ with $\Lambda \nearrow +\infty$, one (formally) recovers a result of Lieb \cite[Theorem 6.2]{Lieb-MultiDimBL} (see also \cite[Corollary 8.16]{BCCT-BrascampLieb}), in which $f_i \in L^1(\R^{n_i})$ are arbitrary. See also \cite[Theorem 2.1]{BezNakamura-RegularizedBrascampLieb} for an inverse Brascamp--Lieb inequality for functions generated by $g_{Q_i}$. 
In addition, Valdimarsson established in \cite[Theorem 1.4 (2)]{Valdimarsson-GenCaffarelli} an analogous \emph{reverse} Brascamp--Lieb inequality for the sup-convolution of functions $f_i$ which are more log-concave than $g_{Q_i}$, extending Barthe's proof from \cite{Barthe-ReverseBL}; as already mentioned, when $P_i(x) = x_i$, this reverse form coincides with the case $Q=0$ and $Q_i > 0$ of Theorem \ref{thm:intro-IBL2}, but without requiring the evenness assumption.
In view of these analogies (which seem to hint at a natural duality relation between the forward and inverse directions), the requirement that $h_i$ be log-concave in Theorem \ref{thm:intro-IBL} is very natural.

\subsection{Gaussian Correlation Inequality}

Our main motivation for insisting on the extended version of Theorem \ref{thm:intro-IBL} is that it leads to a new and simple proof of Royen's Gaussian Correlation Inequality \cite{Royen-GaussianCorrelation} (see also \cite{LatalaMatlak-GaussianCorrelation}): 

\begin{thm}[Gaussian Correlation Inequality \cite{Royen-GaussianCorrelation}] \label{thm:intro-GCI}
Let $X = (X_1,\ldots,X_N)$ denote a centered Gaussian random vector in $\R^N$ (with arbitrary positive-definite covariance). Then for all $n_1 + n_2 = N$:
\[
\P(\max_{1 \leq i \leq N} |X_i| \leq 1) \geq \P(\max_{1 \leq i \leq n_1} |X_i| \leq 1) \; \P(\max_{n_1+1 \leq i \leq N} |X_i| \leq 1) . 
\]
\end{thm}
The case $\min(n_1,n_2)=1$ above was previously established independently by Khatri \cite{Khatri} and \v{S}id\'ak \cite{Sidak}; see also \cite{SSZ-GaussianCorrelationConjecture,HargeGCCForEllipsoid, CorderoMassTransportAndGaussianInqs} for additional prior results and \cite{Tehranchi-RefinedGaussianCorrelation,ENT-GaussianMixtures,Royen-ImprovedGaussianCorrelation,ACS-RefinedKhatriSidak} for extensions. In fact, Royen established this for a more general family of Gamma distributions, but we restrict our attention to the Gaussian case, for which the non-negative correlation was originally investigated by Pitt \cite{Pitt-GaussianCorrelationInPlane} and remained open (in general) for several decades. Since the inequality holds for arbitrary dimension $N$ and arbitrary covariance, it is well-known and easy to verify (see e.g.~\cite{SSZ-GaussianCorrelationConjecture,LatalaMatlak-GaussianCorrelation} or Subsection \ref{subsec:monotonicity}) that it is equivalent to any of the following statements. Recall that $P_i : \R^N \rightarrow \R^{n_i}$ is given by $P_i x = x_i$ for $x = (x_1,x_2) \in \R^{n_1} \oplus \R^{n_2}$. \begin{itemize}
\item For all origin-symmetric convex sets $K=-K$ and $L=-L$ in $\R^n$ and all $n \geq 1$:
\[
\gamma^n(K \cap L) \geq \gamma^n(K) \gamma^n(L) ,
\]
where $\gamma^n$ denotes the standard Gaussian measure on $\R^n$. This formulation was established for $n=2$ by Pitt \cite{Pitt-GaussianCorrelationInPlane}. 
\item For all even quasi-concave functions $f_1,f_2 : \R^{N} \rightarrow \R_+$, namely functions whose super-level sets $\{ f_i \geq a\}$ are origin-symmetric convex sets for all $a$, one has:
\[
\E f_1(X) f_2(X) \geq \E f_1(X) \E f_2(X) .
\]
\item For all even quasi-concave functions $f_i : \R^{n_i} \rightarrow \R_+$, $i=1,2$, one has:
\[ \E f_1(P_1 X) f_2(P_2 X) \geq \E f_1(P_1 X) \E f_2(P_2 X) .
\] \end{itemize}

\medskip

To see that Theorem \ref{thm:intro-IBL} yields Theorem \ref{thm:intro-GCI}, denote by $\Sigma$, $\Sigma_1$ and $\Sigma_2$ the covariance matrices of $X$, $P_1 X$ and $P_2 X$, respectively. Applying Theorem \ref{thm:intro-IBL} with $\bar Q = \Sigma^{-1}$, $Q_1 = \Sigma_1^{-1}$, $Q_2 = \Sigma_2^{-1}$, $c_1 = c_2 = 1$, $m=2$ and the even log-concave $h_i(x_i) = \mathbf{1}_{K_i}(x_i)$ with $K_i = [-1,1]^{n_i}$, we immediately obtain:
\begin{align}
\nonumber  & \frac{\P(\max_{1 \leq i \leq N} |X_i| \leq 1) }{\P(\max_{1 \leq i \leq n_1} |X_i| \leq 1) \; \P(\max_{n_1+1 \leq i \leq N} |X_i| \leq 1)} \\
\label{eq:equality}  & = \frac{ \E \mathbf{1}_{K_1}(P_1 X) \mathbf{1}_{K_2}(P_2 X)  }{\E \mathbf{1}_{K_1}(P_1 X) \E \mathbf{1}_{K_2}(P_2 X)} \geq  \inf_{A_1,A_2 \geq 0} \frac{\E g_{A_1}(P_1 X) g_{A_2}(P_2 X)}{\E g_{A_1}(P_1 X) \E g_{A_2}(P_2 X)} \\
\label{eq:inf0} & = \inf_{A_1,A_2 \geq 0} \brac{\frac{\det(\Id_{n_1} +  A_1 \Sigma_1) \det(\Id_{n_2} + A_2 \Sigma_2) }{ \det (\Id_{N} + (A_1 \oplus A_2) \Sigma)}}^{\frac{1}{2}}  \\
\label{eq:inf} & =  \inf_{A_1,A_2 \geq 0} \brac{\frac{\det(\Id_{n_1} +  \sqrt{A_1} \Sigma_1 \sqrt{A_1}) \det(\Id_{n_2} + \sqrt{A_2} \Sigma_2 \sqrt{A_2}) }{ \det (\Id_{N} + \sqrt{A_1 \oplus A_2} \Sigma \sqrt{A_1 \oplus A_2})}}^{\frac{1}{2}} ,
\end{align}
where in the last transition we used the obvious fact that $\det(\Id_n + A B) = \det(\Id_n + BA)$ for all $n\times n$ matrices $A,B$. 

It remains to show that the infimum in (\ref{eq:inf}) is precisely $1$, attained at $A_1 = A_2 = 0$. This follows from the following inequality of Fischer \cite[Theorem 7.8.5]{HornJohnson-MatrixAnalysis}, which is an equivalent version of a classical inequality of Hadamard \cite[Theorem 7.8.1]{HornJohnson-MatrixAnalysis}:
\begin{lemma}[Fischer's inequality] \label{lem:Fischer}
Let $M = \begin{pmatrix} M_1 & M_o\\ M_o^* & M_2 \end{pmatrix}$ denote a positive-definite Hermitian matrix in block form. Then:
\begin{equation} \label{eq:Fischer}
\det(M) \leq \det(M_1) \det(M_2) ,
\end{equation}
with equality if and only if $M_o = 0$. 
\end{lemma}
\noindent
Writing $\Sigma = \begin{pmatrix} \Sigma_1 & \Sigma_o \\ \Sigma_o^* & \Sigma_2 \end{pmatrix}$ and applying (\ref{eq:Fischer}) to the positive-definite matrix
\begin{equation} \label{eq:Fischer-matrix}
\Id_N + \sqrt{A_1 \oplus A_2} \Sigma \sqrt{A_1 \oplus A_2} = \begin{pmatrix} \Id_{n_1} + \sqrt{A_1} \Sigma_1\sqrt{A_1} & \sqrt{A_1} \Sigma_o \sqrt{A_2} \\ \sqrt{A_2} \Sigma_o^* \sqrt{A_1} & \Id_{n_2} + \sqrt{A_2} \Sigma_2 \sqrt{A_2} \end{pmatrix} ,
\end{equation}
this concludes the proof of Theorem \ref{thm:intro-GCI}. 
The above simple argument for showing that the infimum in (\ref{eq:inf}) is equal to $1$ is due to an anonymous referee, whom we would like to thank. Our original argument relied on a monotonicity lemma established by Royen in \cite{Royen-GaussianCorrelation} -- see Section \ref{sec:compare} for a comparison between our approach and Royen's one, including analysis of the equality case in Theorem \ref{thm:intro-GCI}.

\smallskip

The above reasoning in fact shows that the log-concavity assumption in Theorem \ref{thm:intro-IBL} cannot be omitted. Otherwise, the same proof would imply that:
\[
\P( P_1 X \in K \wedge P_2 X \in L) \geq \P(P_1 X \in K) \; \P(P_2 X \in L) ,
\]
for any origin-symmetric  $K=-K \subset \R^{n_1}$ and $L=-L \subset \R^{n_2}$ regardless of convexity. Applying this to both $L$ and its complement and summing, it would follow that we must actually have equality above. But clearly it is not true that
\[
\P( P_1 X \in K \wedge P_2 X \in L) = \P(P_1 X \in K) \; \P(P_2 X \in L) 
\]
for general covariance matrices $\Sigma > 0$ even when $K$ and $L$ slabs. Consequently, we deduce (for example) that Theorem \ref{thm:intro-IBL} is false for $N=2$, $m=2$, $n_1=n_2=1$, $h_1 = \textbf{1}_{[-1,1]}$, $h_2 = \textbf{1}_{\R \setminus [-1,1]}$ (which is not log-concave), $\bar Q = \Sigma^{-1}$, $Q_i = \Sigma_{i}^{-1}$ and $c_1=c_2=1$, for generic positive-definite $2 \times 2$ matrices $\Sigma$.

\bigskip
\noindent
\textbf{Acknowledgements}. I thank Dario Cordero-Erausquin, Shohei Nakamura, Shay Sadovsky and Hiroshi Tsuji for very helpful references. I also thank the anonymous referees for their careful reading of the paper and helpful comments and suggestions.

\section{Proof of Theorem \ref{thm:intro-IBL2}} \label{sec:IBL}

The proof of Theorem \ref{thm:intro-IBL2} is a verbatim repetition of the proof of \cite[Theorem 1.3]{NakamuraTsuji-InverseBrascampLieb} by Nakamura--Tsuji. 
Note that we will use the formulation of Theorem \ref{thm:intro-IBL2} rather than the one of the equivalent Theorem \ref{thm:intro-IBL}. Consequently, even if one is only interested in the case $\bar Q \geq 0$ of Theorem \ref{thm:intro-IBL}, we must handle $Q$ of arbitrary signature in Theorem \ref{thm:intro-IBL2} to make the equivalence valid.

Let us now sketch the proof of Theorem \ref{thm:intro-IBL2}, emphasizing the necessary adjustments for handling general $Q_i \geq 0$. Given a sequence of positive semi-definite $n_i \times n_i$ matrices $A_i \geq 0$, $i=1,\ldots,m$, denote $\A = (A_1,\ldots,A_m)$, writing $\A \geq 0$. When $A_i = \lambda \Id_{n_i}$ for all $i=1,\ldots,m$, we will write $\A = \lambda \Id$. We will also denote $\A+\B = (A_i + B_i)_{i=1,\ldots,m}$, and write $\A > 0$ or $\A \leq \B$ if $A_i > 0$ or $A_i \leq B_i$ for all $i=1,\ldots,m$, respectively. Given $0 \leq \A \leq \B$, denote  by $\F_{\A}$ the collection of $\f = (f_1,\ldots,f_m)$ so that $f_i \in L^1(\R^{n_i}, \R_+)$ is even, non-zero, and more log-concave than $g_{A_i}$ for all $i=1,\ldots,m$, and by $\F_{\A,\B} \subset \F_\A$ the subfamily so that $f_i$ is in addition more log-convex than $g_{B_i}$. Also denote by $\G_\A \subset \F_\A$ (respectively, $\G_{\A,\B} \subseteq \F_{\A,\B}$) the subfamilies consisting of centered Gaussians, namely $\g = (g_1,\ldots,g_m)$ where $g_i = g_{C_i}$ with $C_i > 0$ and $C_i \geq A_i$ (respectively, $A_i \leq C_i \leq B_i$). 

\smallskip

Now fix $Q$ and $c_i > 0$ as in Theorem \ref{thm:intro-IBL2}, set:
\[
\BL(\f) := \frac{\int e^{-\frac{1}{2}\scalar{Q x,x}} \Pi_{i=1}^m f_i(x_i)^{c_i} dx}{\Pi_{i=1}^m \brac{\int_{\R^{n_i}} f_i(x_i) dx_i}^{c_i} } ,
\]
and denote:
\[
I_{\A,\B} := \inf_{\f \in \F_{\A,\B}} \BL(\f) ~,~
I_{\A} := \inf_{\f \in \F_{\A}} \BL(\f) ~,~
 I^{\G}_{\A,\B} := \inf_{\g \in \G_{\A,\B}} \BL(\g) ~,~
 I^{\G}_{\A} := \inf_{\g \in \G_{\A}} \BL(\g) .
\]
The proof of Theorem \ref{thm:intro-IBL2} consists of the following four steps. Fix $0 < \A \leq \B$ (we may assume that $\B = \Lambda \Id$). 
\begin{enumerate}
\item \label{it:step1} The infimum in the definition of $I_{\A,\B}$ is attained on some $\f \in \F_{\A,\B}$. 
\item \label{it:step2} If $f \in \F_{\A,\B}$ then 
\[
\BL(\f)^2 \geq I_{\A,\B} \, \BL(\Conv \f),
\]
where $\Conv \f = (\Conv f_1,\ldots,\Conv f_m)$ and $\Conv f_i$ denotes $2^{n_i/2} (f_i \ast f_i)(\sqrt{2} \; \cdot)$. 
\item \label{it:step3} $I_{\A,\B} \geq I^\G_{\A,\B}$ and hence $I_{\A,\B} = I^\G_{\A,\B}$. 
\item \label{it:step4} For all $\Q \geq 0$, $I_\Q = \lim_{\lambda \rightarrow 0^+} \lim_{\Lambda \rightarrow \infty} I_{\Q + \lambda \Id, \Lambda \Id}$. 
\end{enumerate}

These steps were carried out in \cite{NakamuraTsuji-InverseBrascampLieb} for $\A = \lambda \Id$, $\B = \Lambda \Id$ and $\Q = \lambda_0 \Id$, $0 \leq \lambda_0 < \lambda \leq \Lambda < \infty$. The only required modifications to the proof for handling general $\A ,\B,\Q$ as above are encapsulated in the following simple:

\begin{lemma} \label{lem:silly}
Let $f_i \in L^1(\R^n,\R_+)$, $i=1,2$.
\begin{enumerate}[(i)]
\item \label{it:part1} If $f_i$ is more log-concave (respectively, more log-convex) than $g_{A_i}$,  $A_i > 0$ and $i=1,2$, then for all $x \in \R^n$, $\R^n \ni y \mapsto f_1(\frac{x+y}{\sqrt{2}}) f_2(\frac{x-y}{\sqrt{2}})$ is more log-concave (respectively, more log-convex) than $g_{\frac{A_1+A_2}{2}}$.
\item \label{it:part2} If $f_i$ is more log-concave (respectively, more log-convex) than $g_{A_i}$, $A_i > 0$ and $i=1,2$, then $f_1 \ast f_2$ is more log-concave (respectively, more log-convex) than $g_{C}$ for $C^{-1} = A_1^{-1} + A_2^{-1}$. 
In particular, if  $f_1$ is more log-convex than $g_{A_1}$, then without any further assumptions on $f_2$, $f_1 \ast f_2$ is also more log-convex than $g_{A_1}$. 
\end{enumerate}
\end{lemma}
\noindent
Note that Lemma \ref{lem:silly} (\ref{it:part2}) verifies that if $\f \in \F_{\A,\B}$ then $\Conv \f \in \F_{\A,\B}$. 
For $A_i = \lambda_i \Id_n$, $\lambda_i > 0$, Lemma \ref{lem:silly} is proved in \cite[Proposition 2.4 and Lemma 6.2]{NakamuraTsuji-InverseBrascampLieb}. 
\begin{proof}
If $f_i$ is more log-concave (respectively, more log-convex) than $g_{A_i}$, then after translation and scaling, clearly $f_i(\frac{x \pm \cdot}{\sqrt{2}})$ is more log-concave (respectively, more log-convex) than $g_{A_i/2}$, and hence their product is more log-concave (respectively, more log-convex) than $g_{(A_1+A_2)/2}$, establishing part (\ref{it:part1}). 

Part (\ref{it:part2}) is an immediate consequence of a result of Brascamp--Lieb \cite[Theorem 4.3]{BrascampLiebPLandLambda1}, see \cite[Theorem 3.7(b)]{StatisticianSurveyOnStrongLogConcavity} for an explicit proof of the log-concave statement. 
The ``in particular" part follows formally by taking $A_2 = \Lambda \cdot \Id_n$ and sending $\Lambda \nearrow +\infty$. To verify this rigorously, if $f_1 = g_{A_1} h_1$ with $h_1$ log-convex then:
\[
e^{\frac{1}{2} \scalar{A_1 x,x}} (f_1 \ast f_2)(x) = \int_{\R^n} h_1(x-y) f_2(y) e^{\scalar{x, A_1 y} -\frac{1}{2} \scalar{A_1 y,y}} dy .
\]
For each $y$, the integrand is log-convex in $x$, and so by H\"older's inequality the integral is also log-convex in $x$, concluding the proof. 
\end{proof}

Let us now go over the steps in \cite{NakamuraTsuji-InverseBrascampLieb} and check the required modifications. Step (\ref{it:step1}) is proved in \cite[Theorem 2.1]{NakamuraTsuji-InverseBrascampLieb} for $\A = \lambda \Id$ and $\B = \Lambda \Id$ with $0 < \lambda \leq \Lambda < \infty$, and the proof, based on compactness by means of the Arzel\`a--Ascoli theorem, immediately extends to general $0 < \A \leq \B$. 

Step (\ref{it:step2}) is proved in \cite[Proposition 2.4]{NakamuraTsuji-InverseBrascampLieb} for $\A = \lambda \Id$ and $\B = \Lambda \Id$ with $0 < \lambda \leq \Lambda < \infty$. Let us verify the extension to general $0 < \A \leq \B$. Indeed, if $\f \in \F_{\A,\B}$, we may assume without loss of generality that $\int_{\R^{n_i}} f_i(x_i) dx_i = 1$. Set
\[
F(x) := e^{-\frac{1}{2} \scalar{Q x ,x}} \Pi_{i=1}^m f_i(x_i)^{c_i} ,
\]
and compute using the change of variables $u = \frac{x+y}{\sqrt{2}}$ and $v = \frac{x-y}{\sqrt{2}}$:
\begin{align*}
& \BL(\f)^2 = \brac{\int_{\R^N} F(x) dx}^2 = \int_{\R^N} \int_{\R^N} F(u) F(v) du dv  \\
& = \int_{\R^N} \int_{\R^N} e^{-\frac{1}{2} \scalar{Q u,u}}  e^{-\frac{1}{2} \scalar{Q v,v}} \Pi_{i=1}^m \brac{f_i(u_i) f_i(v_i)}^{c_i} du dv \\
& = \int_{\R^N} e^{-\frac{1}{2} \scalar{Q x,x}} \int_{\R^N} e^{-\frac{1}{2} \scalar{Q y,y}} \Pi_{i=1}^m \brac{f_i\brac{\frac{x_i+y_i}{\sqrt{2}}} f_i\brac{\frac{x_i-y_i}{\sqrt{2}}}}^{c_i} dy dx .
\end{align*}
By Lemma \ref{lem:silly} (\ref{it:part1}), for each $x \in \R^{N}$, $\f_x := \brac{f_i\brac{\frac{x_i+\cdot}{\sqrt{2}}} f_i\brac{\frac{x_i-\cdot}{\sqrt{2}}}}_{i=1,\ldots,m}$ belongs to $\F_{\A,\B}$, and so we may proceed as follows:
\begin{align*}
& \geq I_{\A,\B} \int_{\R^N} e^{-\frac{1}{2} \scalar{Q x,x}} \Pi_{i=1}^m \brac{\int_{\R^{n_i}}  f_i\brac{\frac{x_i+y_i}{\sqrt{2}}} f_i\brac{\frac{x_i-y_i}{\sqrt{2}}} dy_i}^{c_i} dx \\
& = I_{\A,\B} \int_{\R^N} e^{-\frac{1}{2} \scalar{Q x,x}}  \Pi_{i=1}^m \Conv f_i(x_i)^{c_i} dx = \I_{\A,\B} \, \BL(\Conv \f) ,
\end{align*}
where the first equality follows by definition of $\Conv f_i$ and a change of variables, and the second since $\Conv f_i$ all integrate to $1$. 

Step (\ref{it:step3}) is proved in \cite[Theorem 2.3]{NakamuraTsuji-InverseBrascampLieb} for $\A = \lambda \Id$ and $\B = \Lambda \Id$ with $0 < \lambda \leq \Lambda < \infty$. Let us verify the extension to general $0 < \A \leq \B$. Indeed, the infimum in the definition of $I_{\A,\B}$ is attained on some minimizing $\f \in \F_{\A,\B}$ by step (\ref{it:step1}). Hence, by step (\ref{it:step2}),
\[
\I_{\A,\B}^2 = \BL(\f)^2 \geq \I_{\A,\B} \BL(\Conv \f) .
\]
By Lemma \ref{lem:silly} (\ref{it:part2}), $\Conv \f \in \F_{\A,\B}$, and hence
\[
\I_{\A,\B} \geq \BL(\Conv \f) \geq \I_{\A,\B} ,
\]
so $\Conv \f$ is also a minimizer. Iterating this, we deduce that $\I_{\A,\B} = \BL(\Conv^k \f)$ with $\Conv^k \f \in \F_{\A,\B}$ for all $k \geq 1$. Invoking the Central Limit Theorem as in \cite{NakamuraTsuji-InverseBrascampLieb}, it follows that each  $\Conv^k f_i$ converges in $L^1(\R^{n_i})$ and almost everywhere as $k \rightarrow \infty$ to a centered Gaussian $\frac{1}{Z_i} g_{C_i} \in \F_{\A,\B}$. By Fatou's lemma, denoting $\g = (g_{C_1},\ldots,g_{C_m})$, 
\[
\I_{\A,\B} = \lim_{k \rightarrow \infty} \BL(\Conv^k \f) \geq \BL(\g) \geq \I^\G_{\A,\B}. 
\]
Since the converse inequality is trivial, we verify that $\I_{\A,\B} = \I^\G_{\A,\B}$. 

Step (\ref{it:step4}) is proved in \cite[Lemma 2.7]{NakamuraTsuji-InverseBrascampLieb} for $\Q = \lambda_0 \Id$, $\lambda_0 \geq 0$. Let us verify the extension to general $\Q \geq 0$. Given $\f \in \F_\Q$, the idea is to define
\[
 \f_{\lambda,\Lambda} := ((f_i g_{2 \lambda \Id_{n_i}}) \ast g_{\Lambda \Id_{n_i}})_{i=1,\ldots,m}.
\]
Note that $f_i g_{2 \lambda \Id_{n_i}}$ is more log-concave than $g_{Q_i + 2 \lambda \Id_{n_i}}$, and so by Lemma \ref{lem:silly} (\ref{it:part2}), by choosing $\Lambda$ large enough (in a manner depending only on $\Q$ and $\lambda$), $(f_i g_{2 \lambda \Id_{n_i}}) \ast g_{\Lambda \Id_{n_i}}$ is more log-concave than $g_{Q_i + \lambda \Id_{n_i}}$, and in addition more log-convex than $g_{\Lambda \Id_{n_i}}$. Consequently, $\f_{\lambda, \Lambda} \in \F_{\Q + \lambda \Id , \Lambda \Id}$, from whence the argument in \cite{NakamuraTsuji-InverseBrascampLieb} may be repeated verbatim. 

\smallskip

Combining steps (\ref{it:step3}) and (\ref{it:step4}), we conclude that:
\[
I_\Q = \lim_{\lambda \rightarrow 0^+} \lim_{\Lambda \rightarrow \infty} I_{\Q + \lambda \Id, \Lambda \Id} = \lim_{\lambda \rightarrow 0^+} \lim_{\Lambda \rightarrow \infty} I^\G_{\Q + \lambda \Id , \Lambda \Id} \geq I^\G_{\Q} . 
\]
But since $I^\G_{\Q} \geq I_\Q$ we must have equality, concluding the proof of Theorem \ref{thm:intro-IBL2}.

\section{Comparison with Royen's approach} \label{sec:compare}

Recall that $X$ in Theorem \ref{thm:intro-GCI} denotes a centered Gaussian random vector in $\R^N$, whose (positive-definite) covariance matrix $\Sigma$ is written in block form as 
\begin{equation} \label{eq:Sigma-block}
\Sigma = \begin{pmatrix} \Sigma_1 & \Sigma_o \\ \Sigma_o^* & \Sigma_2 \end{pmatrix} 
\end{equation}
with $\Sigma_i = P_i \Sigma P_i^*$, $i=1,2$, where $P_i x = x_i$ for $x = (x_1,x_2) \in \R^{n_1} \oplus \R^{n_2} = \R^N$.

\smallskip
Royen's proof that $\E \textbf{1}_{K_1}(P_1 X) \textbf{1}_{K_2}(P_2 X) \geq \E \textbf{1}_{K_1}(P_1 X) \E \textbf{1}_{K_2}(P_2 X)$ 
seems to be rather specific to the case that $K_i$ are cubes (or rectangles), but on the other hand applies to the case when $(X_i^2)_{i=1,\ldots,N}$ is distributed according to more general Gamma distributions. On the other hand, our proof remains unchanged if one replaces $\mathbf{1}_{K_i}$ by any two other even log-concave functions $f_i : \R^{n_i} \rightarrow \R_+$, but is restricted to the Gaussian case (and of course {\it a posteriori}, the correlation inequality ultimately holds for any two even quasi-concave functions $f_i$). Let us now take a more in depth look at the advantages of each approach.

\subsection{The equality case in Theorem \ref{thm:intro-GCI}}

In \cite{Royen-GaussianCorrelation}, Royen in fact showed that equality occurs in Theorem \ref{thm:intro-GCI} if and only if $ (X_1,\ldots,X_{n_1})$ and $(X_{n_1+1},\ldots,X_N)$ are independent, or equivalently, if and only if $\Sigma_o = 0$. Let us give an alternative proof of this following our approach. 

\begin{proposition}
Equality occurs in Theorem \ref{thm:intro-GCI} if and only if $\Sigma_o = 0$. 
\end{proposition}
\begin{proof}
The ``if" direction is trivial, so assume that we have equality in Theorem \ref{thm:intro-GCI}. This means that the left- and right-hand sides of (\ref{eq:equality}) coincide and are equal to $1$, and so employing the notation used in Section \ref{sec:IBL}, this means that $f_i = \textbf{1}_{[-1,1]^{n_i}} g_{\Sigma_i^{-1}}$ is an extremizer of $\BL(\f)$ over $f \in \F_{\Q}$ with $\Q = (\Sigma_1^{-1},\Sigma_2^{-1})$. By step (\ref{it:step3}) from Section \ref{sec:IBL}, there also exist centered Gaussian extremizers $\g = (g_1,g_2) \in \G_{\Q}$ with $\BL(\g) = \BL(\f) = 1$. Since the covariance of $f_i$ does not change under each iteration of the self-convolution operator $\Conv^{k} f_i$, this is also preserved in the limit as $k \rightarrow \infty$ and hence $g_i$ and $f_i$ have the same covariance matrices. Consequently, $g_i = g_{\Sigma_i^{-1} + A_i}$ for some positive-definite $A_i > 0$, $i=1,2$ (as opposed to merely positive semi-definite). Since the ratio in (\ref{eq:inf}) is equal to $1$ for these particular $A_i$, we have equality in Fischer's inequality, and we deduce from Lemma \ref{lem:Fischer} and (\ref{eq:Fischer-matrix}) that $\sqrt{A_1} \Sigma_o \sqrt{A_2} = 0$. But since $A_i$ are positive-definite, we conclude that $\Sigma_o =0$, as asserted. 
\end{proof}

\subsection{Original argument}

Our original proof that the infimum in (\ref{eq:inf0}) is equal to $1$ proceeded by rewriting (\ref{eq:inf0}) as:
\begin{equation} \label{eq:inf2}
 \inf_{A_1,A_2 \geq 0} \brac{\frac{\det(\Id_{N} +  (A_1 \oplus A_2) (\Sigma_1 \oplus \Sigma_2))}{ \det (\Id_{N} + (A_1 \oplus A_2) \Sigma)}}^{\frac{1}{2}} ,
\end{equation}
and employing the following:

\begin{lemma} \label{lem:monotone}
Let $\Sigma$ be a $N \times N$ positive-definite matrix, given in block form by (\ref{eq:Sigma-block}). 
Denote
\begin{equation} \label{eq:Sigmat}
 \Sigma^{(s)} := (1-s) (\Sigma_1 \oplus \Sigma_2) + s \Sigma = \begin{pmatrix} \Sigma_1 & s \Sigma_o \\ s \Sigma_o^* & \Sigma_2 \end{pmatrix} ~,~ s \in [0,1] . 
\end{equation}
 Then for all $n_i \times n_i$ positive semi-definite matrices $A_i \geq 0$, $i=1,2$, the function $[0,1] \ni s \mapsto \det(\Id_N + (A_1 \oplus A_2) \Sigma^{(s)})$ is monotone non-increasing. In particular, the infimum in (\ref{eq:inf2}) is $1$. 
\end{lemma}

By choosing an appropriate orthonormal basis, one easily reduces to the case where $A_i$ are diagonal. In that case, Lemma \ref{lem:monotone} was established by Royen for his proof of Theorem \ref{thm:intro-GCI} in \cite{Royen-GaussianCorrelation}.
Curiously, in \cite{Royen-GaussianCorrelation}, the diagonal elements of $A_1 \oplus A_2$ represent the coefficients of the Laplace transform of the distribution of $(X_i^2/2)_{i=1,\ldots,N}$, whereas in our case, the $A_i^{-1}$'s correspond to the covariances of the Gaussians which potentially saturate the correlation inequality. For completeness, we reproduce the  detailed proof in \cite{LatalaMatlak-GaussianCorrelation}
following \cite{Royen-GaussianCorrelation} below. 

\begin{proof}[Proof of Lemma \ref{lem:monotone}]
Denote $[N] = \{1,\ldots,N\}$, and for $J \subset [N]$, denote by $B_J$ is the principal minor of an $N\times N$ matrix $B$ with rows and columns in $J$. Also denote $|B| = \det(B)$ for brevity. For any pair of orthogonal transformations $U_i \in O(n_i)$,
\[
|\Id_N + (A_1 \oplus A_2) \Sigma^{(s)}| = |\Id_N + (U_1^{-1} A_1 U_1 \oplus U_2^{-1} A_2 U_2) (U_1 \oplus U_2)^{-1} \Sigma^{(s)} (U_1 \oplus U_2) | .
\]
Choosing appropriate $U_i$ and redefining $A_i$, $\Sigma_i$ and $\Sigma$, we may assume that $A_i$ are diagonal. Denote $A := A_1 \oplus A_2$, a diagonal matrix. Then by \cite[Lemma 5 (i)]{LatalaMatlak-GaussianCorrelation}:
\[
|\Id_N + A \Sigma^{(s)}| = 1 + \sum_{\emptyset \neq J \subset [N]} |(A \Sigma^{(s)})_J| = 1 + \sum_{\emptyset \neq J \subset [N]} |A_J| |\Sigma^{(s)}_J| .
\]
Since $A \geq 0$ and hence $|A_J| \geq 0$, it is enough to show that $[0,1] \ni s \mapsto |\Sigma^{(s)}_J|$ is non-increasing for each $\emptyset \neq J \subset [N]$. Denote $J_1 = J \cap [n_1]$ and $J_2 = J \setminus [n_1]$. If $J_1 = \emptyset$ or $J_2 = \emptyset$ then $\Sigma^{(s)}_J \equiv \Sigma_J$ remains constant. Otherwise, write:
\[
\Sigma^{(s)}_J = \begin{pmatrix} \Sigma_{J_1} & s \Sigma_{J_1 J_2} \\ s \Sigma_{J_2 J_1} & \Sigma_{J_2} \end{pmatrix} ,
\]
and observe by \cite[Lemma 5 (ii)]{LatalaMatlak-GaussianCorrelation} that:
\[
|\Sigma^{(s)}_J| = |\Sigma_{J_1}| |\Sigma_{J_2}| |\Id_{J_1} - s^2 M| ~,~ M := \Sigma_{J_1}^{-1/2} \Sigma_{J_1 J_2} \Sigma_{J_2}^{-1} \Sigma_{J_2 J_1} \Sigma_{J_1}^{-1/2} ,
\]
with $0 \leq M \leq \Id_{J_1}$. It follows that $[0,1] \ni s \mapsto |\Sigma^{(s)}_J|$ is non-increasing, concluding the proof. 
\end{proof}

\subsection{Monotonicity} \label{subsec:monotonicity}

Royen actually showed that $[0,1] \ni s \mapsto \E \mathbf{1}_{[-1,1]^N}(X^{(s)})$ is monotone non-decreasing, where $X^{(s)}$ denotes a centered Gaussian random vector with covariance $\Sigma^{(s)}$ defined in (\ref{eq:Sigmat}). 
While our proof (employing Lemma \ref{lem:monotone} instead of Lemma \ref{lem:Fischer}) only shows that this is lower-bounded by a monotone non-decreasing function, it elucidates the reason for the correlation -- the worst case is given by $\E g_{A_1 \oplus A_2}(X^{(s)})$, which is explicitly computable and turns out to be monotone in $s$. 

Royen's result is equivalent to the statement that $[0,1] \ni s \mapsto \E f_1(Y_s) f_2(Z_s)$ is monotone non-decreasing for any even quasi-concave functions $f_1,f_2 : \R^n \rightarrow \R_+$ and a family of centered Gaussian random vectors $(Y_s,Z_s)$ in $\R^{n} \times \R^n$ having covariance matrix
\begin{equation} \label{eq:form}
\begin{pmatrix} \Sigma_{YY} & s \Sigma_{YZ} \\ s \Sigma_{YZ}^* & \Sigma_{ZZ} \end{pmatrix} . 
\end{equation}
Indeed, by a standard layer-cake decomposition of $f_1$ and $f_2$ into their super-level sets, the claim reduces to the case when $f_i = \textbf{1}_{K_i}$ with $K_i$ an origin-symmetric convex set in $\R^n$. To see the monotonicity in that case, approximate $K_i$ by an origin-symmetric polytope $\cap_{j=1}^{n_i} \{ x \in \R^n : |\sscalar{x,a^i_j}| \leq 1\}$, and arrange $\{a^i_j\}_{j=1,\ldots,n_i}$ as the rows of an $n_i \times n$ matrix $A_i$. Treating all random vectors as column vectors, define the centered Gaussian random-vector $X^{(s)} = (A_1 Y_s + \eps G_1, A_2 Z_s + \eps G_2)$ in $\R^N$, where $N = n_1+n_2$, and $G_i$ is a standard Gaussian random vector in $\R^{n_i}$ which is independent of all other random vectors. 
It follows that the (non-degenerate) covariance matrix of $X^{(s)}$ is given by
\[
\begin{pmatrix} A_1 \Sigma_{YY} A_1^* + \eps^2 \Id_{n_1} & s A_1 \Sigma_{YZ} A_2^* \\ s A_2 \Sigma_{YZ}^* A_1^* & A_2 \Sigma_{ZZ} A_2^* + \eps^2 \Id_{n_2} \end{pmatrix} .
\]
Royen's result implies that $[0,1] \ni s \mapsto \E \mathbf{1}_{[-1,1]^N}(X^{(s)})$ is monotone, and the monotonicity is preserved in the limit after refining the approximation of the $K_i$'s and sending $\eps \rightarrow 0$, thereby establishing the monotonicity of $[0,1] \ni s \mapsto \E \mathbf{1}_{K_1}(Y_s) \mathbf{1}_{K_2}(Z_s)$.

It is worthwhile noting that this monotonicity has a useful consequence for the Ornstein--Uhlenbeck semi-group $P_t$ associated to the standard Gaussian measure $\gamma^n$ on $\R^n$, defined for (say) $f \in L^2(\gamma^n)$ as
\[
P_t f(y) := \E f (e^{-t} y + \sqrt{1-e^{-2t}} Z_0) ~,~ Z_0 \sim \gamma^n ~,~ t \in [0,\infty]~,~ y \in \R^n . 
\]
Note that $P_0 f = f$ and $P_\infty f \equiv \int f d\gamma^n$. Consequently, the Gaussian Correlation Inequality may be equivalently reformulated as stating that for any even quasi-concave functions $f_1,f_2 : \R^n \rightarrow \R_+$, one has:
\[
\int f_1 P_0 f_2  d\gamma^n \geq \int f_1 P_\infty f_2 d\gamma^n . 
\]
Royen's argument actually implies the stronger statement that $[0,\infty] \ni t \mapsto \int f_1 P_t f_2 d\gamma^n$ is monotone non-increasing; when $n=2$, this was previously established by Pitt \cite{Pitt-GaussianCorrelationInPlane}. To see this, note that
\[
\int f_1 P_t f_2 d\gamma^n = \E(f_1(Y) f_2 (Z_s))  ~ ,~ Z_s = s Y + \sqrt{1-s^2} Z_0 ~,~ s = e^{-t} ,
\]
where $Y$ and $Z_0$ are independent standard Gaussian random vectors in $\R^n$. Consequently, the covariance matrix of the random vector $(Y,Z_s)$ in $\R^{n} \times \R^n$ is of the form (\ref{eq:form}) with $\Sigma_{YY} = \Sigma_{YZ} = \Sigma_{ZZ} = \Id_n$, and the monotonicity follows. 

In contrast, we do not know how to obtain the full monotonicity of $P_t$ using the inverse Brascamp--Lieb approach.

\subsection{Subsequent developments}

After posting and submitting our work, Nakamura and Tsuji noted that there is actually no need to assume that the functions $f_i = g_{Q_i} h_i$ in Theorem \ref{thm:intro-IBL2} are even, and that in order to run the Central Limit argument in step (\ref{it:step3}) of the proof in Section \ref{sec:IBL}, it is only necessary that their barycenter is at the origin. After treating various additional technicalities which arise, they were able to show in \cite{NakamuraTsuji-GCIForCentered} that:
\begin{equation} \label{eq:final}
\gamma^n(K \cap L) \geq \gamma^n(K) \gamma^n(L) 
\end{equation}
for all convex $K,L \subset \R^n$ whose Gaussian barycenter is at the origin (``centered"), including a characterization of the equality case. In the origin-symmetric case, a characterization of the equality case can also be deduced from a stability version of (\ref{eq:final}) obtained in \cite[Theorem 21 and Proposition 23]{DNS-StabilityOfCorrelationInqs-ConferenceVer}. 
We refer to \cite{NakamuraTsuji-GCIForCentered} for further extensions of (\ref{eq:final}) to the case when the Gaussian barycenters of $K$ and $L$ coincide (but differ from $0$), and for a correlation result for more than two centered convex sets $K_1,\ldots,K_m$. 

In contrast, it seems that Royen's approach is inherently restricted to origin-symmetric sets and even functions, because the first step in his proof is to consider the squared random vector $(X_i^2)_i$, and its distribution can only capture information on the sigma-algebra generated by events of the form $X_i \in [-r_i,r_i]$.


\begin{thebibliography}{10}

\bibitem{ACS-RefinedKhatriSidak}
R.~Assouline, A.~Chor, and S.~Sadovsky.
\newblock A refinement of the \v{S}id\'ak-{K}hatri inequality and a strong
  {G}aussian correlation conjecture.
\newblock arXiv:2407.15684, 2024.

\bibitem{Barthe-ReverseBL-CRAS}
F.~Barthe.
\newblock In\'egalit\'es de {B}rascamp-{L}ieb et convexit\'e.
\newblock {\em C. R. Acad. Sci. Paris S\'er. I Math.}, 324(8):885--888, 1997.

\bibitem{Barthe-ReverseBL}
F.~Barthe.
\newblock On a reverse form of the {B}rascamp-{L}ieb inequality.
\newblock {\em Invent. Math.}, 134(2):335--361, 1998.

\bibitem{BartheCordero-InverseBLviaSemiGroup}
F.~Barthe and D.~Cordero-Erausquin.
\newblock Inverse {B}rascamp-{L}ieb inequalities along the heat equation.
\newblock In {\em Geometric aspects of functional analysis}, volume 1850 of
  {\em Lecture Notes in Math.}, pages 65--71. Springer, Berlin, 2004.

\bibitem{BCLM-BrascampLiebInqs}
F.~Barthe, D.~Cordero-Erausquin, M.~Ledoux, and B.~Maurey.
\newblock Correlation and {B}rascamp-{L}ieb inequalities for {M}arkov
  semigroups.
\newblock {\em Int. Math. Res. Not. IMRN}, (10):2177--2216, 2011.

\bibitem{BartheHuet}
F.~Barthe and N.~Huet.
\newblock On {G}aussian {B}runn-{M}inkowski inequalities.
\newblock {\em Studia Math.}, 191(3):283--304, 2009.

\bibitem{BartheWolff-InverseBrascampLieb}
F.~Barthe and P.~Wolff.
\newblock Positive {G}aussian kernels also have {G}aussian minimizers.
\newblock {\em Mem. Amer. Math. Soc.}, 276(1359):v+90, 2022.

\bibitem{BCCT-BrascampLieb}
J.~Bennett, A.~Carbery, M.~Christ, and T.~Tao.
\newblock The {B}rascamp-{L}ieb inequalities: finiteness, structure and
  extremals.
\newblock {\em Geom. Funct. Anal.}, 17(5):1343--1415, 2008.

\bibitem{BezNakamura-RegularizedBrascampLieb}
N.~Bez and S.~Nakamura.
\newblock Regularized {B}rascamp--{L}ieb inequalities.
\newblock arXiv:2110.02841, to appear in Analysis \& PDE, 2021.

\bibitem{Boroczky-BrascampLiebSurvey}
K.~J. B\"{o}r\"{o}czky.
\newblock The {B}rascamp-{L}ieb inequality in convex geometry and in the theory
  of algorithms.
\newblock arXiv:2412.11227, 2024.

\bibitem{BrascampLieb-YoungInq}
H.~J. Brascamp and E.~H. Lieb.
\newblock Best constants in {Y}oung's inequality, its converse, and its
  generalization to more than three functions.
\newblock {\em Advances in Math.}, 20(2):151--173, 1976.

\bibitem{BrascampLiebPLandLambda1}
H.~J. Brascamp and E.~H. Lieb.
\newblock On extensions of the {B}runn-{M}inkowski and {P}r\'ekopa-{L}eindler
  theorems, including inequalities for log concave functions, and with an
  application to the diffusion equation.
\newblock {\em J. Func. Anal.}, 22(4):366--389, 1976.

\bibitem{BrenierMap}
Y.~Brenier.
\newblock Polar factorization and monotone rearrangement of vector-valued
  functions.
\newblock {\em Comm. Pure Appl. Math.}, 44(4):375--417, 1991.

\bibitem{CaffarelliContraction}
L.~A. Caffarelli.
\newblock Monotonicity properties of optimal transportation and the {FKG} and
  related inequalities.
\newblock {\em Comm. Math. Phys.}, 214(3):547--563, 2000.

\bibitem{CarlenCordero-SubadditivityOfEntropy}
E.~A. Carlen and D.~Cordero-Erausquin.
\newblock Subadditivity of the entropy and its relation to {B}rascamp-{L}ieb
  type inequalities.
\newblock {\em Geom. Funct. Anal.}, 19(2):373--405, 2009.

\bibitem{CarlenLiebLoss-EntropyOnSn}
E.~A. Carlen, E.~H. Lieb, and M.~Loss.
\newblock A sharp analog of {Y}oung's inequality on {$S^N$} and related entropy
  inequalities.
\newblock {\em J. Geom. Anal.}, 14(3):487--520, 2004.

\bibitem{CDP-ReverseHolder}
W.-K. Chen, N.~Dafnis, and G.~Paouris.
\newblock Improved {H}\"older and reverse {H}\"older inequalities for
  {G}aussian random vectors.
\newblock {\em Adv. Math.}, 280:643--689, 2015.

\bibitem{ChewiPooladian-GenCaffarelli}
S.~Chewi and A.-A. Pooladian.
\newblock An entropic generalization of {C}affarelli's contraction theorem via
  covariance inequalities.
\newblock {\em C. R. Math. Acad. Sci. Paris}, 361:1471--1482, 2023.

\bibitem{CorderoMassTransportAndGaussianInqs}
D.~Cordero-Erausquin.
\newblock Some applications of mass transport to {G}aussian-type inequalities.
\newblock {\em Arch. Ration. Mech. Anal.}, 161(3):257--269, 2002.

\bibitem{CourtadeLiu-BrascampLieb}
T.~A. Courtade and J.~Liu.
\newblock Euclidean forward-reverse {B}rascamp-{L}ieb inequalities: finiteness,
  structure, and extremals.
\newblock {\em J. Geom. Anal.}, 31(4):3300--3350, 2021.

\bibitem{DNS-StabilityOfCorrelationInqs-ConferenceVer}
A.~De, S.~Nadimpalli, and R.~A. Servedio.
\newblock {Quantitative Correlation Inequalities via Semigroup Interpolation}.
\newblock In {\em 12th Innovations in Theoretical Computer Science Conference
  (ITCS 2021)}, volume 185 of {\em Leibniz International Proceedings in
  Informatics (LIPIcs)}, pages 69:1--69:20, 2021.

\bibitem{ENT-GaussianMixtures}
A.~Eskenazis, P.~Nayar, and T.~Tkocz.
\newblock Gaussian mixtures: entropy and geometric inequalities.
\newblock {\em Ann. Probab.}, 46(5):2908--2945, 2018.

\bibitem{GozlanSylvestre-GeneralizedContractions}
N.~Gozlan and M.~Sylvestre.
\newblock Global regularity estimates for optimal transport via entropic
  regularisation.
\newblock arxiv:2501.11382, 2025.

\bibitem{HargeGCCForEllipsoid}
G.~Harg{\'e}.
\newblock A particular case of correlation inequality for the {G}aussian
  measure.
\newblock {\em Ann. Probab.}, 27(4):1939--1951, 1999.

\bibitem{HornJohnson-MatrixAnalysis}
R.~A. Horn and C.~R. Johnson.
\newblock {\em Matrix analysis}.
\newblock Cambridge University Press, Cambridge, second edition, 2013.

\bibitem{Khatri}
C.~G. Khatri.
\newblock On certain inequalities for normal distributions and their
  applications to simultaneous confidence bounds.
\newblock {\em Ann. Math. Statist.}, 38:1853--1867, 1967.

\bibitem{KolesnikovContractionSurvey}
A.~V. Kolesnikov.
\newblock Mass transportation and contractions.
\newblock arXiv:1103.1479, 2011.

\bibitem{LatalaMatlak-GaussianCorrelation}
R.~Lata{\l}a and D.~Matlak.
\newblock Royen's proof of the {G}aussian correlation inequality.
\newblock In {\em Geometric aspects of functional analysis}, volume 2169 of
  {\em Lecture Notes in Math.}, pages 265--275. Springer, Cham, 2017.

\bibitem{Lieb-MultiDimBL}
E.~H. Lieb.
\newblock Gaussian kernels have only {G}aussian maximizers.
\newblock {\em Invent. Math.}, 102(1):179--208, 1990.

\bibitem{LCCV-BrascampLieb}
J.~Liu, T.~A. Courtade, P.~W. Cuff, and S.~Verd\'u.
\newblock A forward-reverse {B}rascamp-{L}ieb inequality: entropic duality and
  {G}aussian optimality.
\newblock {\em Entropy}, 20(6):Paper No. 418, 32, 2018.

\bibitem{NakamuraTsuji-HypercontractivityBeyondNelson}
S.~Nakamura and H.~Tsuji.
\newblock Hypercontractivity beyond {N}elson's time and its applications to
  {B}laschke--{S}antal\'o inequality and inverse {S}antal\'o inequality.
\newblock arXiv:2212.02866, 2022.

\bibitem{NakamuraTsuji-VolumeProductUnderHeatFlow}
S.~Nakamura and H.~Tsuji.
\newblock The functional volume product under heat flow.
\newblock arXiv:2401.00427, 2024.

\bibitem{NakamuraTsuji-InverseBrascampLieb}
S.~Nakamura and H.~Tsuji.
\newblock A generalized {L}egendre duality relation and {G}aussian saturation.
\newblock arXiv:2409.13611, 2024.

\bibitem{NakamuraTsuji-GCIForCentered}
S.~Nakamura and H.~Tsuji.
\newblock The {G}aussian correlation inequality for centered convex sets and
  the case of equality.
\newblock arXiv:2504.04337, 2025.

\bibitem{Pitt-GaussianCorrelationInPlane}
L.~D. Pitt.
\newblock A {G}aussian correlation inequality for symmetric convex sets.
\newblock {\em Ann. Probability}, 5(3):470--474, 1977.

\bibitem{Royen-GaussianCorrelation}
T.~Royen.
\newblock A simple proof of the {G}aussian correlation conjecture extended to
  some multivariate gamma distributions.
\newblock {\em Far East J. Theor. Stat.}, 48(2):139--145, 2014.

\bibitem{Royen-ImprovedGaussianCorrelation}
T.~Royen.
\newblock On improved {G}aussian correlation inequalities for symmetrical
  $n$-rectangles extended to certain multivariate {G}amma distributions and
  some further probability inequalities.
\newblock {\em Far East J. Theor. Stat.}, 69(1):1–38, 2024.

\bibitem{StatisticianSurveyOnStrongLogConcavity}
A.~Saumard and J.~A. Wellner.
\newblock Log-concavity and strong log-concavity: a review.
\newblock {\em Stat. Surv.}, 8:45--114, 2014.

\bibitem{SSZ-GaussianCorrelationConjecture}
G.~Schechtman, Th. Schlumprecht, and J.~Zinn.
\newblock On the {G}aussian measure of the intersection.
\newblock {\em Ann. Probab.}, 26(1):346--357, 1998.

\bibitem{Tehranchi-RefinedGaussianCorrelation}
M.~R. Tehranchi.
\newblock Inequalities for the {G}aussian measure of convex sets.
\newblock {\em Electron. Commun. Probab.}, 22:Paper No. 51, 7, 2017.

\bibitem{Valdimarsson-GenCaffarelli}
S.~I. Valdimarsson.
\newblock On the {H}essian of the optimal transport potential.
\newblock {\em Ann. Sc. Norm. Super. Pisa Cl. Sci. (5)}, 6(3):441--456, 2007.

\bibitem{Sidak}
Z.~\v{S}id\'ak.
\newblock Rectangular confidence regions for the means of multivariate normal
  distributions.
\newblock {\em J. Amer. Statist. Assoc.}, 62:626--633, 1967.

\end{thebibliography}

\def\cprime{$'$} \def\textasciitilde{$\sim$}

\end{document}